\documentclass[12pt]{amsart}

\usepackage{mathrsfs}
\usepackage{graphicx}
\usepackage{amsmath}
\usepackage{amssymb}
\usepackage{amsthm}
\usepackage{graphicx}

\newtheorem{theorem}{Theorem}
\theoremstyle{plain}

\newtheorem{proposition}[theorem]{Proposition}

\newcommand{\bE}{\mathbb{E}}
\newcommand{\bR}{\mathbb{R}}

\newcommand{\mfu}{\mathfrak{u}}
\def\mfz{\mathfrak{z}}
\def\mfh{\mathfrak{h}}

\newcommand\bel[1]{\begin{equation}\label{#1}}
\newcommand\ee{\end{equation}}
\newcommand\bld[1]{\boldsymbol{#1}}

\def\wc{\overset{d}{=}}
\def\wcL{\overset{{\mathcal{L}}}{=}}
\def\wth{\widehat{\theta}}
\def\dX{\dot{X}}

\def\vs{\vskip 0.1in}

\def\oas{0_{\mathrm{a.s.}}}

\numberwithin{equation}{section}
\numberwithin{theorem}{section}

\begin{document}

\today

\title[Estimation in linear CAR(2)]{Second-order continuous-time non-stationary
Gaussian autoregression}
\author{N. Lin}
\curraddr[N. Lin]{Department of Mathematics, USC\\
Los Angeles, CA 90089 USA\\
tel. (+1) 213 821 1480; fax: (+1) 213 740 2424}
\email[N. Lin]{nlin@usc.edu}
\author{S. V. Lototsky}
\curraddr[S. V. Lototsky (corresponding author)]
{Department of Mathematics, USC\\
Los Angeles, CA 90089 USA\\
tel. (+1) 213 740 2389; fax: (+1) 213 740 2424}
\email[S. V. Lototsky]{lototsky@usc.edu}
\urladdr{http://www-rcf.usc.edu/$\sim$lototsky}


 \begin{abstract}
 The objective of the paper is to identify and investigate all possible
 types of asymptotic behavior for the maximum likelihood estimators of
 the unknown parameters in the second-order linear stochastic ordinary differential
 equation driven by Gaussian white noise. The emphasis is on the non-ergodic case, when
 the roots of the corresponding characteristic equation are not both in the left half-plane.
\end{abstract}

\maketitle

\section{Introduction}

Consider the stochastic ordinary differential  equation
\begin{equation}
\label{eq00}
\ddot{X}(t)=\theta_1\dot{X}(t)+\theta_2 X(t)+\sigma\dot{W}(t),\ t>0,
\end{equation}
with a standard Brownian motion $W=W(t)$,  non-random initial conditions
$X(0)$, $\dX(0)$, and two real parameters $\theta_1, \theta_2$.
 Equation \eqref{eq00} is often referred to as a continuous time
auto-regression of second order, or CAR(2), being a particular case of
 CAR(N)
\begin{equation}
\label{CAR(N)}
X^{(N)}=\sum_{k=0}^{N-1} \theta_{N-k}X^{(k)}+\dot{W};
\end{equation}
see \cite{BDY-CAR}.
A  rigorous  interpretation of equation \eqref{eq00} is the
system
\begin{equation}
\label{eq000}
dX=\dot{X}dt, \ d\dot{X}=(\theta_2 X+\theta_1\dot{X})dt +\sigma dW(t).
\end{equation}
 In the matrix-vector form,  system \eqref{eq000}  becomes a
particular case of the multi-di\-men\-sional Ornstein-Uhlenbeck process studied in \cite{BasakLee}:
\begin{equation}
\label{OU-Md}
d\bld{X}(t)=\Theta \bld{X}(t) dt + \bld{\sigma} dW(t),
\end{equation}
with
$$
\bld{X}(t)=
\left(
\begin{array}{l}
X(t)\\
\dX(t)
\end{array}
\right),\
\Theta=
\left(
\begin{array}{ll}
0 &1\\
 \theta_2 &  \theta_1
 \end{array}
 \right), \
\bld{\sigma}=
\left(
\begin{array}{l}
0\\
\sigma
\end{array}
\right).
 $$
The estimator studied in \cite{BasakLee} is
\begin{equation}
\label{BL}
\widehat{\Theta}_T=\left(\int_0^T\big(d\bld{X}(t) \bld{X}^{\top}(t)\big)\right)
\left(\int_0^T \bld{X}(t)\bld{X}^{\top}(t)dt\right)^{-1}.
\end{equation}
It is strongly consistent as $T\to \infty$, and this implies strong consistency of the
maximum likelihood estimators for $\theta_1$ and $\theta_2$ in
\eqref{eq00} for all $(\theta_1, \theta_2)\in \bR^2$; see \cite{BasakLee} or Theorem \ref{th-consist} below.

When the process $\bld{X}$ defined by \eqref{OU-Md} is ergodic
(equivalently, when all eigenvalues of the matrix
$\Theta$ are in the left half-plane),
it is known \cite[Theorem 4.6.2]{Arato}  that the
 estimator $\widehat{\Theta}_T$ is asymptotically normal
 with rate $T^{1/2}$.
  Some non-ergodic models of the type \eqref{OU-Md} have also been
 studied (see, for example,  \cite{Luschgy-LAMN}). The key object in the
analysis of a maximum likelihood estimator  (MLE) is  the normalized
log-likelihood ratio. A desirable property of this normalized
log-likelihood ratio is local asymptotic normality (LAN), as it implies
certain  efficiency of the  MLE; see \cite[Chapter II]{IKh}.
While the LAN property is typically associated with ergodic models,
it can also hold in some non-ergodic models
\cite{KhasJan}. Still, in many non-ergodic models, LAN is replaced with
LAMN (local asymptotic mixed normality), leading to a different kind of
 efficiency of the MLE; see \cite[Chapter 5]{LeCamY}.

 So far, analysis of \eqref{OU-Md} in general and \eqref{eq00} in particular
was aimed at identifying the particular cases that fit in either LAN or LAMN framework. It turns out that many of the non-ergodic regimes of \eqref{eq00}
lead to new asymptotic forms of the normalized log-likelihood ratio
and to new types of asymptotic behavior of the MLE. In the current paper, we
present a complete asymptotic analysis of  the
MLE and the likelihood ratio for all possible values of the
unknown parameters $(\theta_1, \theta_2)\in \bR^2$.
Beside purely theoretical interest,
this analysis can help in the investigation of the corresponding discrete-time models (see \cite{BDY-CAR} in the ergodic case).

The rest of the paper is organized as follows. Section \ref{S2}
  states the main results  and discusses the results in the broader context of statistical estimation.
  After some preliminary work in Section \ref{Sec4}, the proofs are in
Section \ref{S4}, followed by a brief summary in Section \ref{S5}.

We fix a stochastic basis
 $(\Omega, \mathcal{F}, \{\mathcal{F}_t\}_{t\geq 0}, \mathbb{P})$ with
a standard Brownian motion $W=W(t)$. Other common notations are
 $\mathbb{E}$
 for  the expectation with respect to $\mathbb{P}$,
 $\sqrt{-1}$ for the imaginary unit,
 ${}^{\top}$ for the transpose of a vector or a matrix,
  $\dot{g}(t)$, $\ddot{g}(t)$ for the first and second time
  derivatives of the function $g$,
 $\wc$ for  equality in distribution of random variables,
 $\wcL$ for equality in law of random processes,
 and $\oas(T)$ to denote a process converging to zero with
probability one as $T\to \infty$.

\section{Summary of the main results}

To write the likelihood ratio for equation \eqref{eq00}, define the vectors
\begin{equation}
\label{eq:vec}
 \bld{\theta}=
 \left(
 \begin{array}{l}
 \theta_2\\
 \theta_1
 \end{array}
 \right),\ \
\bld{X}(t)=
\left(
 \begin{array}{l}
X(t)\\
\dX(t)
 \end{array}
 \right).
 \end{equation}
Then equation \eqref{eq00} becomes
\begin{equation}
\label{eqv}
d\dX=\bld{\theta}^{\top}\bld{X}dt+\sigma dW(t).
\end{equation}
Since $X(t)=X(0)+\int_0^t \dX(s)ds$, it follows from \eqref{eqv} that
 \begin{enumerate}
 \item  $\dX=\dX(t)$ is a diffusion-type process in the sense of Liptser and
Shiryaev (\cite[Definition 4.2.7]{LSh1}) and generates the
measure $P^{\bld{\theta}}_T$ on the space of continuous functions on $[0,T]$,
\item  For every $\bld{\theta}\in \bR^2$, the
  measure $ P^{\bld{\theta}}_T$ is absolutely continuous with respect to the
  measure $P^{\bld{0}}_T$ generated by the process
  $\dX(0)+ \sigma W(t),\ 0\leq t\leq T$, and
\begin{equation}
\begin{split}
\label{LR}
&\frac{dP^{\bld{\theta}}_T}{dP^{\bld{0}}_T}(\dX)\\
&=
\exp\left(\frac{1}{\sigma^2} \int_0^T (\theta_2 X(t)+\theta_1\dX(t))d\dX(t)
-\frac{1}{2\sigma^2}\int_0^T(\theta_2\,X(t)+\theta_1\dX(t))^2dt\right);
\end{split}
\end{equation}
see  \cite[Theorem 7.6]{LSh1}.
\end{enumerate}
The  expressions for the maximum likelihood estimators $\wth_{1,T}$ of $\theta_1$ and $\wth_{2,T}$ of $\theta_2$,
using  continuous-time observations of
 both $X(t)$, and $\dot{X}(t)$, $0\leq t\leq T$, now follow from
 \eqref{LR}:
 \begin{equation}
 \label{MLE-main}
 \begin{split}
 \wth_{1,T}&=\frac{\int_0^T X^2(t)dt\int_0^T \dX(t) d\dX(t)
 -\int_0^TX(t)\dX(t)dt\int_0^TX(t)d\dX(t)}
 {\int_0^TX^2(t)dt \int_0^T \dX^2(t)dt-\left(\int_0^TX(t)\dX(t)dt\right)^2},\\
 \wth_{2,T}&=\frac{\int_0^T \dX^2(t)dt\int_0^T X(t) d\dX(t)
 -\int_0^TX(t)\dX(t)dt\int_0^T\dX(t)d\dX(t)}
 {\int_0^TX^2(t)dt \int_0^T \dX^2(t)dt-\left(\int_0^TX(t)\dX(t)dt\right)^2}.
 \end{split}
 \end{equation}
 Similar to \eqref{eq:vec} we write
\begin{equation}
\label{vec-e}
\widehat{\bld{\theta}}_T=
 \left(
 \begin{array}{l}
 \wth_{2,T}\\
 \wth_{1,T}
 \end{array}
 \right).
 \end{equation}

The amount of  integration in \eqref{MLE-main}
can be reduced using rules of the usual and stochastic calculus and
keeping in mind that the processes $X$ is continuously differentiable and
 the process $\dot{X}$ is
a continuous semi-martingale with  quadratic variation equal to $\sigma^2t$:
\begin{equation}
\label{integrals}
\begin{split}
\int_0^T{X}(t)\dot{X}(t)dt&=\int_0^TX(t)dX(t)=\frac{X^2(T)-X^2(0)}{2}, \\
\int_0^T \dot{X}(t)d\dot{X}(t)&=\frac{\dot{X}^2(T)-\sigma^2 T-\dX^2(0)}{2},\\
\int_0^TX(t)d\dot{X}(t)&=X(T)\dot{X}(T)-X(0)\dX(0)-\int_0^T \dot{X}^2(t)dt.
\end{split}
\end{equation}
 With the continuous time observations, the value of
 $\sigma$ can be assumed known because the quadratic
 variation process of $\dot{X}$ at time $t$, an observable quantity, is $\sigma^2t$. Note also that $\sigma$ does not appear in the formulas
 \eqref{MLE-main}.

 The measure  $P^{\bld{0}}_T$ in \eqref{LR} corresponds to the solution of
 \eqref{eqv} with $\theta_1=\theta_2=0$.
 Sometimes it is more convenient to work with
  the reference measure coming from some other solution of
  \eqref{eq00}, for example, the actual observations.  Accordingly, let us fix $\bld{\theta}\in \bR^2$ and
 the corresponding observation process
 $\bld{X}=\bld{X}(t),\ 0\leq t\leq T,$
 satisfying \eqref{eq000}.
 Then, for every $\bld{\vartheta}\in \bR^2$,
 \begin{equation}
 \label{LLR}
\begin{split}
 \ln \frac{dP^{\bld{\vartheta}}_T}{dP^{\bld{\theta}}_T}(\dX)&=
 \frac{1}{\sigma^2} \int_0^T \big( \bld{\vartheta}^{\top}\bld{X}(t)-
  \bld{\theta}^{\top}\bld{X}(t)\big)d\dX(t)\\
&- \frac{1}{2\sigma^2}
\int_0^T \Big(\big(\bld{\vartheta}^{\top}\bld{X}(t)\big)^2
-\big(\bld{\theta}^{\top}\bld{X}(t)\big)^2\Big)dt;
 \end{split}
 \end{equation}
 see \cite[Theorem 7.19]{LSh1}.
 With $\bld{\theta}$ and $\bld{X}$ fixed, we write
 $$
 L_T(\vartheta)=
 \ln \frac{dP^{\bld{\vartheta}}_T}{dP^{\bld{\theta}}_T}(\dX).
 $$
 Define the matrix
\begin{equation}
\label{FI}
\Psi_T=\int_0^T \bld{X}(t)\bld{X}^{\top} (t) dt =
\left(
\begin{array}{ll}
\int_0^T X^2(t)dt & \int_0^T X(t)\dX(t) dt\\
 & \\
\int_0^T X(t)\dX(t) dt & \int_0^T \dX^2(t) dt
\end{array}
\right),
\end{equation}
 Then \eqref{eqv} and \eqref{LLR} imply
 \begin{align}
 \notag
 L_T(\vartheta)&=
 \frac{1}{\sigma}
  \int_0^T \big( \bld{\vartheta}-\bld{\theta}\big)^{\top}\bld{X}(t)dW(t)
- \frac{1}{2\sigma^2}
\int_0^T \big(\bld{\vartheta}-\bld{\theta}\big)^{\top}\bld{X}(t)\bld{X}^{\top}(t)
\big(\bld{\vartheta}-\bld{\theta}\big) dt\\
&
\label{LLR1}
= \frac{1}{\sigma}
  \int_0^T \big( \bld{\vartheta}-\bld{\theta}\big)^{\top}\bld{X}(t)dW(t)
- \frac{1}{2\sigma^2}
 \big(\bld{\vartheta}-\bld{\theta}\big)^{\top}\Psi(T)
\big(\bld{\vartheta}-\bld{\theta}\big).
 \end{align}

We address the following  questions about the
 estimators $\wth_1(T), \wth_2(T)$:
 \begin{enumerate}
 \item {\bf rate of convergence}, that is, finding positive deterministic functions
 $v_i(T)$, $i=1,2$, such that, as $T\to \infty$, $v_i(T)\nearrow +\infty$ and
  $ v_i(T)\big(\wth_{i,T}-\theta_i\big)$
   converge in distribution to non-degenerate
  random variables, and identifying the corresponding limit distributions;
  \item existence of a {\bf normal limit with a random rate} (NLRR),
  that is, finding a random matrix $R=R_T$ such that, as
  $T\to \infty$, $R_T(\widehat{\bld{\theta}}_T-\bld{\theta})$ converges in
  distribution to a Gaussian random vector.  The hope is that at least one of the two things happens: (a) the random rate leads to a normal limit when a deterministic rate does not; (b) neither the matrix $R_T$
   nor the parameters of the limit distribution depend explicitly on the $\theta_1, \theta_2$.
   \item {\bf local asymptotic structure} of the normalized log-likelihood
  ratio, that is, analyzing
   \begin{equation}
   \label{LAQ1}
   \ell_T(\bld{u})=L_T(\bld{\theta}+A_T\bld{u}),
   \end{equation}
   as $T\to \infty$, where $\bld{\theta}$ is fixed, 
   $A_T\in \bR^{2\times 2}$ is a suitably chosen
   deterministic  matrix with $\lim_{T\to \infty} |A_T|=0$ 
   (any matrix norm will work), and $\bld{u}\in \bR^2$.
   \end{enumerate}

   Let us recall some definitions related to the normalized
   log-likelihood ratio \eqref{LAQ1}.  It follows from \eqref{LLR1} that
    \begin{equation}
  \label{LAQ}
 \ell_T(\bld{u})= \frac{1}{\sigma}
 \int_0^T\Big(\bld{u}^{\top} A_T \bld{X}(t)\Big)\,dW(t)-
 \frac{1}{2\sigma^2} \bld{u}^{\top}A_T^{\top}\Psi_TA_T\bld{u}.
  \end{equation}
   With a suitable choice of the matrix $A_T$,  there exists a non-trivial limit
   in distribution
   \begin{equation}
   \label{LAQ-lim}
   \ell_{\infty}(\bld{u}) = \lim_{T\to \infty} \ell_T(\bld{u}).
   \end{equation}
   Three particular cases of $\ell_T$ satisfying \eqref{LAQ-lim}
    are of special interest:
  \begin{itemize}
  \item {\tt Local Asymptotic Normality} (LAN), if there exists a bivariate normal vector $\bld{\xi}$ with mean zero and non-degenerate
      covariance matrix
      $\Sigma_{\bld{\xi}}$ such that, for every $\bld{u}\in \bR^2,$
      \begin{equation}
  \label{LAN}
  \ell_{\infty}(\bld{u})=
   \frac{1}{\sigma}\bld{u}^{\top} \bld{\xi}-\frac{1}{2\sigma^2} \bld{u}^{\top}\Sigma_{\bld{\xi}}\bld{u}.
   \end{equation}
   \item {\tt Local  Asymptotic Mixed Normality} (LAMN) if there
   exist a bivariate normal vector $\bld{\eta}$ with zero mean and unit covariance matrix,  and a random symmetric positive definite matrix
   $B\in \bR^{2\times 2}$
   such that  $B$ and $\bld{\eta}$ are independent and,  for every $\bld{u}\in \bR^2$,
  \begin{equation}
  \label{LAMN}
   \ell_{\infty}(\bld{u})=\frac{1}{\sigma}\bld{u}^{\top}B^{1/2}\bld{\eta}- \frac{1}{2\sigma^2}\bld{u}^{\top}B\bld{u}.
  \end{equation}
 If \eqref{LAMN} holds with a degenerate matrix $B$, we refer to $\ell_{T}$ as DLAMN (degenerate locally asymptotically mixed normal).
  \item {\tt Local Asymptotic Brownian Functional} structure (LABF) if
  \begin{equation}
  \label{LABF}   \ell_{\infty}(\bld{u})=\frac{1}{\sigma}\int_0^1\bld{u}^{\top}G(t)d\bld{w}(t)- \frac{1}{2\sigma^2}\int_0^1\bld{u}^{\top}G(t)G^{\top}(t)\bld{u}\,dt,
  \end{equation}
  where $G\in \bR^{2\times 2}$
  is an adapted process, $\bld{w}\in \bR^2$
  is a standard Brownian motion, and the pair $(G,\bld{w})$ is a  Gaussian process.
 \end{itemize}
 The definitions of LAN, LAMN, and LABF extend to more general likelihood
 ratios and to any finite number of unknown parameters. For details see
 \cite[Chapter II]{IKh} (LAN), \cite[Chapter 1]{BasawaScott} (LAMN),
 and \cite[Section 2]{Jeg-Gen} (unified approach to LAN, LAMN, and LABF).

  The  LAN and LAMN properties of $\ell_T$
  imply certain asymptotic efficiency of the corresponding
  maximum likelihood estimator (MLE):
  \begin{enumerate}
  \item If $\ell_T$ is LAN, then the corresponding MLE is
  asymptotically efficient in the sense of achieving the lower bound in the
  Cramer-Rao inequality; for details see \cite[Theorem II.12.1]{IKh}.
  \item If $\ell_T$ is LAMN, then the corresponding MLE has the
  maximal concentration property; for details,
  see \cite[Theorem 2.2.1]{BasawaScott}. Note that the result requires
   non-degeneracy of the matrix $B$ and therefore does not
  immediately extend to DLAMN.
  \end{enumerate}
In the  LABF case, there are results about asymptotic efficiency of Bayessian
 estimators \cite[Section 3, Proposition 10]{Jeg-Gen} and sequential estimators
 \cite[Theorem 2]{Grw-Wfl}.

   To put our results in perspective, let us recall the estimation problem
    of the drift $\theta$ in the CAR(1) model,
which is the one-dimensional OU process $Y=Y(t)$ defined by
\begin{equation}
\label{OU}
dY(t)=\theta Y(t)dt + \sigma dW(t).
\end{equation}

Here is a summary of the results. For details, see \cite{Feigin, Kut2}.
\begin{itemize}
\item The maximum likelihood estimator $\wth_T$ of $\theta$ using  the
observations of $Y(t),\ 0\leq t\leq T$ is
\begin{equation}
 \label{OU-MLE}
 \wth_T=\frac{\int_0^T Y(t)dY(t)}{\int_0^TY^2(t)dt};
 \end{equation}
 the estimator is strongly consistent as $T\to \infty$:
 $\lim_{T\to\infty} \wth_T=\theta$ with probability one for all
 $\theta\in \mathbb{R}$.
 \item If $\theta<0$ (asymptotically stable or ergodic case), then
 \bel{OU-A1}
 \lim_{T\to \infty} \sqrt{{|\theta|T}}(\wth_T-\theta) \wc \sqrt{2}\,|\theta|\, \xi,
 \ee
 where $\xi$ is a standard normal random variable.
 \item If $\theta=0$ (neutrally stable case), then
 \bel{OU-A2}
 \lim_{T\to \infty} T(\wth_T-\theta) \wc \frac{w^2(1)-1}{2\int_0^1w^2(s)ds},
 \ee
  where $w=w(s),\ 0\leq s\leq 1,$ is a standard Brownian motion.
 \item If $\theta>0$ (unstable or explosive case), then
 \bel{OU-A3}
 \lim_{T\to \infty}{e^{\theta T}}
  (\wth_T-\theta) \wc {2\theta}\,\frac{\eta}{\xi+c},
  \ee
  where $\xi=\sqrt{2\theta} \int_0^{\infty} e^{-\theta t}dW(t)$ is a
  standard normal random variable,
  $\eta$ is a standard normal random variable independent of $\xi$,
  and $c=\sqrt{2\theta}\,Y(0)/\sigma$. In particular, if $Y(0)=0$, then
  the limit has the Cauchy distribution and does not depend on $\sigma$.
  \item If $\theta\not=0$, then NLRR holds:
  \bel{OU-A4}
  \lim_{T\to \infty} \left(\int_0^T Y^2(t)dt\right)^{1/2} \big(\wth_T-\theta\big)
   \wc \sigma \,\eta,
  \ee
where $\eta$ is a standard normal random variable.
  \item The normalized log-likelihood ratio
  $$
 \ell_T(u)=
\frac{u}{\sigma}\,\frac{ \int_0^T Y(t)dW(t)}{\left(\bE\int_0^T Y^2(t)dt\right)^{1/2}}-\frac{u^2}{2\sigma^2}\,
\frac{\int_0^TY^2(t)dt}{\bE\int_0^T Y^2(t)dt},\ \ u\in \bR.
$$
is  LAN if  $\theta<0$, LABF if $\theta=0$, and  LAMN if  $\theta>0$.
  \end{itemize}
 Note that (a) the limit distributions in \eqref{OU-A1}, \eqref{OU-A2}, and
 \eqref{OU-A4} do not depend on the initial condition; (b) equality \eqref{OU-A4}
 illustrates the attractive features of NLRR: the rate
 $$
 R(T)=\left(\int_0^T Y^2(t)dt\right)^{1/2}
 $$
 does not explicitly depend on $\theta$ and  the limit distribution does not
  depend on $\theta$ or the initial conditions.

 Table 1 summarizes the results, where  $F_d(w)$  denotes a
 generic functional of the standard $d$-dimensional Brownian motion and $\mathrm{Ch}$ denotes
 a Cauchy-type distribution (ratio of two independent normal random variables).

 \begin{center}
 \begin{table}[ht]
 \label{tab1}
 \caption{Estimation in CAR(1)}
 \begin{tabular}{lcccc}
 Parameter \phantom{$[t]_{H_H}$}& Rate & LD & $\ell_T$& NLRR   \\ \hline

$\theta<0$ \phantom{$\widehat{[t]}^{H}$}&$\sqrt{T} $ & $\mathcal{N}$ & LAN & Yes \\

 $\theta=0$ \phantom{$\widehat{[t]}^{H}$}&$T$        &  $F_1(w)$ & LABF  & No\\
$\theta>0$ \phantom{$\widehat{[t]}^{H}$}& $e^{\theta T}$       &  Ch &
 LAMN & Yes\\
 \end{tabular}
 \end{table}
 \end{center}

Let us now turn to equation \eqref{eq00}.
Asymptotic behavior of estimators \eqref{MLE-main}  depends on
  the roots $p,q$ of the characteristic equation
  \begin{equation}
 \label{CharEq}
 r^2-\theta_1r-\theta_2=0.
 \end{equation}
 There are nine cases to consider:
 \begin{enumerate}
 \item The asymptotically stable (ergodic) case, when $\theta_1<0$ and $\theta_2<0$.
 All in all, there are three possibilities for the roots:
 $q<p<0,\  q=p<0,$ or $p=\lambda+ \sqrt{-1}\,\nu,\ q=\lambda- \sqrt{-1}\,\nu,$ with $\lambda<0$, but the asymptotic behavior
 of the estimators is the same in all three cases.
 \item Six non-ergodic  cases with real $p,q$:
  $q<p=0$, $q<0<p$,
 $p>q=0$, $p>q>0$, $p=q>0$, or $p=q=0$.
 \item The harmonic oscillator, when
 $p= \sqrt{-1}\, \nu, \ q=-\sqrt{-1}\, \nu,\ \nu>0$;
 \item Unstable oscillations, when $p=\lambda+ \sqrt{-1}\,\nu,$\
 $q=\lambda- \sqrt{-1}\,\nu$, and  $\lambda>0, \nu>0$.
 \end{enumerate}

 Table 2 summarizes the results for CAR(2).
 The detailed statements  are in Section \ref{S2}.
 In Table 2, we use the same notations as in Table 1. In particular, $F_d(w)$
 denotes a functional of the standard $d$-dimensional Brownian motion and
 $\mathrm{Ch}$ is a Cauchy-type distribution.

 \begin{center}
 \begin{table}[ht]
 \label{tab4}
 \caption{Estimation in CAR(2)}
 \begin{tabular}{lccccll}
 Case  &$v_1$ & $v_2$ & LD$_1$ & LD$_2$ &  NLRR & $\ell_{T}$ \phantom{$[t]_{H_H}^{H^H}$} \\
  \hline
 $\theta_1<0, \  \theta_2<0$ & $\sqrt{T}$ & $\sqrt{T}$ & $\mathcal{N}$& $\mathcal{N}$& Yes & LAN  \phantom{$[t]_{H_H}^{H^H}$}\\
 $q<0<p$ &  $\sqrt{T}$ & $\sqrt{T}$ & $\mathcal{N}$& $\mathcal{N}$& Yes & DLAMN  \phantom{$[t]_{H_H}^{H^H}$}\\
 $0<q<p$ & $e^{qT}$& $e^{qT}$ & Ch & Ch & Yes & DLAMN  \phantom{$[t]_{H_H}^{H^H}$}\\
 $0<q=p$ & $T^{-1}e^{qT}$ & $T^{-1}e^{qT}$ & Ch & Ch & Yes & DLAMN  \phantom{$[t]_{H_H}^{H^H}$}\\
 $q<p=0$ & $\sqrt{T}$ & $T$ & $\mathcal{N}$ & $F_1(w)$ & Yes $(\wth_{1,T})$ & LABF/LAN \phantom{$[t]_{H_H}^{H^H}$}\\
 $q=0<p$ & $T$ & $T$ & $F_1(w)$ & $F_1(w)$ & No & DLAMN  \phantom{$[t]_{H_H}^{H^H}$}\\
 $q=p=0$ & $T$ & $T^2$ & $F_1(w)$ & $F_1(w)$ & No & LABF  \phantom{$[t]_{H_H}^{H^H}$}\\
 $\Re(p)=0$ & $T$ & $T$ & $F_2(w)$ & $F_2(w)$ & No & LABF  \phantom{$[t]_{H_H}^{H^H}$}\\
$\Re (p)=\lambda>0$ & $e^{\lambda T}$& $e^{\lambda T}$ & Many & Many & Yes & LAMN family  \phantom{$[t]_{H_H}^{H^H}$}\\
 \end{tabular}
 \end{table}
 \end{center}

There are obvious similarities between CAR(1) and CAR(2)
in  the ergodic case and,
 perhaps less obvious, similarities in the neutrally stable case
 ($\theta=0$ in CAR(1) compared to $p=\sqrt{-1}\nu$ in CAR(2))
 and in the exponentially unstable cases ($\theta>0$ in CAR(1) compared to
 real $p,q>0$ in CAR(2)). In both CAR(1) and CAR(2),  the rate
$\sqrt{T}$ corresponds to normal distribution in the limit, exponential
rate corresponds  to a Cauchy-type distribution, and any rate polynomial in $T$
leads to  some functional of the standard Brownian motion.
In the case of the positive double  root  ($p=q>0)$ the rate is
slightly slower than exponential, but the limit distribution is still
of the  Cauchy type. Also of interest are (a) several appearances  of
 DLAMN instead  of LAMN, (b)  an unusual combination of asymptotic
 normality of the estimators and DLAMN (rather than LAN)  of $\ell_{T}$
 when  $q<0<p$, (c)   relative compactness rather than convergence in
 distribution for both the estimators and the normalized log-likelihood  ratio
 when $p=\lambda+\sqrt{-1}\nu$, $\lambda>0$.

 It is instructive to compare CAR(2) with $p=\lambda+\sqrt{-1}\nu$,
$\lambda>0, \nu>0$,  and the
 example considered in \cite[Section 4.1]{Luschgy-LAMN}:
 \begin{equation}
 \label{Arato}
 \left(
 \begin{array}{l}
 dX_1(t)\\
 dX_2(t)
 \end{array}
 \right)
 =
 \left(
 \begin{array}{lr}
 \lambda & -\nu\\
 \nu & \lambda
 \end{array}
 \right)
 \left(
 \begin{array}{l}
 X_1(t)\\
 X_2(t)
 \end{array}
 \right)
 dt
 +
 \left(
 \begin{array}{l}
 dW_1(t)\\
 dW_2(t)
 \end{array}
 \right).
 \end{equation}
 While the eigenvalues of the matrix in \eqref{Arato} are also
  $\lambda \pm \sqrt{-1} \nu$,
 the special structure of the model ensures that the
 normalized local log-likelihood ratio is LAMN and the MLEs of
 $\lambda$ and $\nu$, when
 normalized by $\sqrt{2}\,\lambda e^{\lambda T}$,
 converge to a joint limit (which, for zero initial conditions,
  is the bivariate $t_2$-distribution).

\vskip 0.1in

\section{Asymptotic properties of the MLE and the normalized
log-likelihood ratio}
\label{S2}

Strong consistency of
\eqref{MLE-main} is a consequence of a more general
result by Basak and Lee \cite{BasakLee}. For the sake of
completeness, here are the statement and the proof.

\begin{theorem}
\label{th-consist}
 With probability one,
$\lim_{T\to \infty} \wth_{1,T}=\theta_1$ and
$\lim_{T\to \infty} \wth_{2,T}=\theta_2$ for every
$\bld{\theta}\in \bR^2$.
\end{theorem}

\begin{proof}
With  $\widehat{\Theta}_T$ defined in \eqref{BL}, we find using \eqref{eq00} that
$$
\widehat{\Theta}_T=
\left(
\begin{array}{ll}
0 &1\\
 \wth_{2,T} &  \wth_{1,T}
 \end{array}
 \right).
 $$
The statement of the theorem now follows from
   \cite[Theorem 2.1 and Remark 3.1]{BasakLee}.
\end{proof}

Next, we present the theorem describing the limit distributions of
$\wth_{1,T}$ and $\wth_{2,T}$.
The proofs are in  Section \ref{S4}.
To keep visual track of the formulas, it is convenient to think of the
process $X$ in \eqref{eq00} as  a dimensionless quantity and to measure $t$
 in the units $[t]$ of time.
 Table 3 summarizes the resulting dimensions of all the variables and parameters
  in the problem.
 \vskip 0.1in
 \begin{center}
 \begin{table}[ht]
 \label{tab2}
 \caption{Dimensions in CAR(2)}
 \begin{tabular}{lcccccc}
 Quantity \phantom{$[t]_{H_H}$}&$t$ & $X(t)$&$W(t)$&
 $\dX(t),\,\theta_1,\, \wth_{1,T},\,p,\,q$ & $\sigma$ &
 $\ddot{X}(t),\,\theta_2,\,\wth_{2,T}$ \\
 \hline
 Units \phantom{$\widehat{[t]}^{H}$}&$[t]$ &  None&$[t]^{1/2}$& $[t]^{-1}$
 & $[t]^{-3/2}$ & $[t]^{-2}$
 \end{tabular}
 \end{table}
 \end{center}

 One caveat: the auxiliary Brownian motion $w(s), \ 0\leq s\leq 1$, and its
 parameter $s$, as well as all random variables appearing in the limit
 distributions, are dimensionless.

\begin{theorem}[Rate of convergence and limit distributions]
\label{th2}
${ } $ \\
 {\sc I.  Ergodic case.}  Assume that $\theta_1<0$ and $\theta_2<0$, and
 let $\eta_1, \eta_2$ be  iid standard normal
random variables. Then
\begin{equation}
\label{erg1}
\lim_{T\to \infty} \sqrt{T|\theta_1|}(\wth_{1,T}-\theta_1) \wc
\sqrt{2}|\theta_1|\, \eta_1,\ \
\lim_{T\to \infty} \sqrt{T|\theta_1|}(\wth_{2,T}-\theta_2) \wc
\sqrt{2|\theta_2|}\,|\theta_1|\, \eta_2.
\end{equation}

For the rest of the theorem, denote by $p$ and $q$ the roots of equation \eqref{CharEq}.

{\sc II. Non-ergodic case: distinct real roots.} Let $\xi, \eta$ be iid standard normal
random variables, and let $w=w(s),\ 0\leq s\leq 1,$  be a standard Brownian
motion independent of $\eta$.

{\sc (a)} If $p>0$ and $q<0$, then
\begin{equation}
\label{c2.1}
\lim_{T\to \infty} \sqrt{|q|{T}}\,(\wth_{1,T}-\theta_1)\wc -\frac{1}{p} \lim_{T\to \infty} \sqrt{|q|{T}}\,(\wth_{2,T}-\theta_2)\wc  \sqrt{2}\, |q|\,\eta,
\end{equation}

 {\sc (b)} If  $p>q>0$, then
\begin{equation}
\label{c4.1}
\lim_{T\to \infty}{e^{qT}}(\wth_{1,T}-\theta_1)\wc
-\frac{1}{p} \lim_{T\to \infty} {e^{qT}}\,(\wth_{2,T}-\theta_2)\wc
\frac{2(p+q)q}{p-q}\, \frac{\eta}{\xi+c},
\end{equation}
where $c=\sqrt{2q}\,\big(\dX(0)-pX(0)\big)/\sigma.$

{\sc (c)} If  $p=0$ and $q<0$, then $\theta_1=q$, $\theta_2=0$, and
\begin{equation}
\label{c1.1}
\lim_{T\to \infty} \sqrt{|\theta_1|T}\big(\wth_{1,T}-\theta_1\big)\wc \sqrt{2}\,|\theta_1|\,\eta,\ \
\lim_{T\to \infty} {T}\wth_{2,T}\wc
|\theta_1|\, \frac{w^2(1)-1}{2\int_0^1w^2(s)ds}.
\end{equation}

{\sc (d)} If  $p>0$ and $q=0$. Then $\theta_1=p$, $\theta_2=0$, and
\begin{equation}
\label{c3.1}
\lim_{T\to \infty} \theta_1T\,(\wth_{1,T}-\theta_1)\wc
- \lim_{T\to \infty} T\wth_{2,T}\wc
\theta_1 \, \frac{w^2(1)-1}{2\int_0^1w^2(s)ds}.
\end{equation}

{\sc III. Non-ergodic case: a double root.}

{\sc (a)} If $p=q>0$, then $\theta_1=2q$, $\theta_2=q^2$, and
\begin{equation}
\label{c5.1}
\lim_{T\to \infty}\frac{e^{qT}}{qT}(\wth_{1,T}-\theta_1)\wc
 -\frac{1}{q}\lim_{T\to \infty} \frac{e^{qT}}{qT}\,(\wth_{2,T}-\theta_2)\wc
4\sqrt{2}\,q \frac{\eta}{\xi+c},
\end{equation}
where $c=\sqrt{2q}\,\big(\dX(0)-pX(0)\big)/\sigma$ and $\xi,\,\eta$
are iid standard normal random variables.

{\sc (b)} If $p=q=0$, then $\theta_1=0$, $\theta_2=0$, and
\begin{equation}
\begin{split}
\label{c6.1}
&\lim_{T\to \infty} T\wth_{1,T}\wc
\frac{2\mfz_3\big(w^2(1)-1\big)-2\mfz_1^2\big(w(1)\mfz_1-\mfz_2\big)}{4\mfz_2\mfz_3-\mfz_1^4},\\
&\lim_{T\to \infty} T^2\wth_{2,T}\wc
\frac{4\mfz_2\big(w(1)\mfz_1-\mfz_2\big)-\mfz_1^2\big(w^2(1)-1\big)}{4\mfz_2\mfz_3-\mfz_1^4},
\end{split}
\end{equation}
 where $w=w(s),\ 0\leq s\leq 1,$ is a standard Brownian motion, and
$$
\mfz_1=\int_0^1 w(s)ds,\ \
\mfz_2=\int_0^1 w^2(s)ds,\ \
\mfz_3=\int_0^1\left( \int_0^t w(s)ds\right)^2 dt.
$$

{\sc IV. Non-ergodic case: complex roots.}

{\sc (a)}  If $p=\sqrt{-1}\nu,\ \nu>0$, then $\theta_1=0$, $\theta_2=-\nu^2$,
and
\begin{equation}
\label{HO-LD}
\begin{split}
&\lim_{T\to \infty} T\wth_{1,T}\wc
\frac{2-w_1^2(1)-w_2^2(1)}{\int_0^1w_1^2(t)dt+
\int_0^1w_2^2(t)dt},\\
& \lim_{T\to \infty}T(\widehat{\theta}_{2,T}-\theta_2)\wc 2\nu
\frac{\int_0^1w_1(t)dw_2(t)-\int_0^1w_2(t)dw_1(t)}{\int_0^1w_1^2(t)dt+
\int_0^1w_2^2(t)dt},
\end{split}
\end{equation}
where $w_1, w_2$ are independent standard Brownian motions.

{\sc (b)} If $p=\lambda+\sqrt{-1}\nu$, $\lambda>0,\ \nu>0$, then, for $i=1,2$,
the families $\{e^{\lambda T}(\wth_{i,T}-\theta_i),\ T> 0\}$
are relatively compact, and all the limit distributions are the form
$$
\frac{\bar{\xi}_c\bar{\eta}_c+\bar{\xi}_s\bar{\eta}_s}{\bar{\xi}_c^2+\bar{\xi}_s^2},
$$
where $(\bar{\xi}_c, \bar{\xi}_s)$ and
$(\bar{\eta}_c,\bar{\eta}_s)$ are independent bivariate normal vectors,   $\bE\bar{\eta}_c=\bE\bar{\eta}_s=0$, and the mean values of
$\bar{\xi}_c$ and $\bar{\xi}_s$ depend on the initial
conditions $X(0), \dX(0)$.

\end{theorem}

One general conclusion of Theorem \ref{th2} is that, if $p>0$ and $q<p$,
then it is the value of the
smaller root  $q$ that determines asymptotic behavior of the estimators.
This result comes as a surprise: the asymptotic behavior of both
estimators is dictated by the non-dominant mode, even though this
mode is ``invisible'' with probability one. Indeed,
the solution of \eqref{eq00} is a Gaussian process
\begin{equation}
\label{intr-meq}
X(t)=X(0)x_1(t)+\dX(0)x_2(t)+\sigma\int_0^t x_2(t-s)dW(s),
\end{equation}
 where the functions $x_1(t), x_2(t)$ form the fundamental system of solutions for
 the equation
 \begin{equation}
 \label{DE-main}
 \ddot{x}(t)-\theta_1 \dot{x}(t)-\theta_2 x(t)=0.
 \end{equation}
 In other words, $x_1(0)=1, \dot{x}_1(0)=0$, $x_2(0)=0$,
 $\dot{x}_2(0)=1$, and both $x_1=x_1(t)$ and $x_2=x_2(t)$ satisfy
 \eqref{DE-main}.
 The roots of the characteristic equation \eqref{CharEq} are
 \begin{equation}
 \label{CharRoots}
 p=\frac{\theta_1+\sqrt{\theta_1^2+4\theta_2}}{2},\ \
q=\frac{\theta_1-\sqrt{\theta_1^2+4\theta_2}}{2}.
 \end{equation}
  Then, with the usual modifications for complex, $p,q$,
 \begin{equation}
 \label{FS1}
 x_1(t)=
 \begin{cases}
 \displaystyle \frac{qe^{pt}-pe^{qt}}{q-p},& {\rm if}\  p>q,\\
 (1-qt)e^{qt},& {\rm if } \ p=q;
 \end{cases}
 \quad x_2(t)=
 \begin{cases}
 \displaystyle \frac{e^{pt}-e^{qt}}{p-q},& {\rm if}\  p>q,\\
 te^{qt},& {\rm if } \ p=q.
 \end{cases}
 \end{equation}
 When the roots   \eqref{CharRoots}
are real and distinct, equation  \eqref{DE-main} has two Lyapunov exponents,
$p$ and $q$, and it follows from \eqref{intr-meq} that if $q<p$, then
$$
\mathbb{P}\left(\lim_{t\to \infty} \frac{1}{t} \ln |X(t)| = q\right)=
\mathbb{P}\left(\lim_{t\to \infty} \frac{1}{t} \ln |\dX(t)| = q\right)=0
$$
for all initial conditions $X(0),\  \dX(0)$.  Thus, Theorem \ref{th2} shows that
if the larger Lyapunov exponent of \eqref{DE-main} is positive, then the
asymptotic behavior of the estimators is determined by the
smaller Lyapunov exponent.

As another illustration of the effects of the exponentially unstable mode,
note that if $\theta_2=0$, then  \eqref{eq00} becomes
\bel{c1.3}
d\dX=\theta_1\dX+\sigma \dot{W}.
\ee
In other words,  $\dX$ is a CAR(1) process.
If $\theta_1<0$, then asymptotic behavior of $\wth_{1,T}$
 and $\wth_{2,T}$  and similar to the CAR(1) situation in
 ergodic and neutrally stable cases, respectively. If
 $\theta_2\geq 0$, then there is no clear similarity
 with CAR(1).

 When $p>0$, the normalized limits of
$\wth_{1,T}-\theta_1$ and $\wth_{2,T}-\theta_2$
are negative multiples of each other.  Some correlation is also
present when $p=\lambda+\sqrt{-1}\nu$, $\lambda\geq 0, \nu>0$.
This type of correlation in  non-ergodic multi-parameter
models has been observed before;
see, for example, \cite[Section 4.1]{Luschgy-LAMN}. Still, as the
case $q<0=p$ shows, lack of ergodicity does not necessarily imply
correlation of the limits.

 One can verify qualitative consistency of the results of Theorem \ref{th2} by
 considering various limiting regimes for $p$ and $q$. For example,
 passing to limit $p\searrow q$ in \eqref{c4.1} suggests that
 the rate in the case of the positive double root should be slower than exponential.

\vskip 0.1in

Next, we study the possibility of NLRR, that is,  existence of a random matrix $R=R(T)$ such that
$R(T)(\widehat{\bld{\theta}}_T-\bld{\theta})$ converges in distribution to a bivariate normal
vector.

\begin{theorem}[Normal Limit with a Random Rate (NLRR)]
\label{th3}
Denote by $p$ and $q$ the roots of equation \eqref{CharEq}.
Normal limit with a random rate is possible in the following
six cases: ergodic; $p>q>0$; $p=q>0$; $q<0<p$;
$q<0=p$ (for $\wth_{1,T}$ only); $p=\lambda+\sqrt{-1}\nu$,
$\lambda>0, \ \nu>0$.

 {\sc Ergodic case.}  Assume that $\theta_1<0$, $\theta_2<0$, and let $\eta_1, \eta_2$ be  iid standard normal random variables. Then
\begin{equation}
\label{erg-nlrr}
\begin{split}
&\lim_{T\to \infty} \left(\int_0^T \dX^2(t)dt\right)^{1/2}
(\wth_{1,T}-\theta_1) \wc \sigma \,  \eta_1,\\
&\lim_{T\to \infty}\left(\int_0^T X^2(t)dt\right)^{1/2}(\wth_{2,T}-\theta_2) \wc \sigma\, \eta_2.
\end{split}
\end{equation}

For the rest of the theorem, denote by $\eta$ a standard normal random variable.

\vskip 0.1in

{\sc District Positive Roots.} Assume that $p>q>0$ and define
\begin{equation}
\label{rp}
r(T)= \left(\int_0^T\big(\dX(t)-pX(t)\big)^2dt\right)^{1/2}.
\end{equation}
Then
\begin{equation}
\label{c4.1-nlrr}
\lim_{T\to \infty}r(T)(\wth_{1,T}-\theta_1)\wc-\frac{1}{p} \lim_{T\to \infty} r(T)(\wth_{2,T}-\theta_2)\wc
\frac{p+q}{p-q}\sigma\,\eta.
\end{equation}

{\sc  Positive Double Root.}
Assume that $p=q>0$ and  define
$$
r(T)= \frac{1}{T^2}\left(\int_0^TX^2(t)dt\right)^{1/2}.
$$
Then
\begin{equation}
\label{c4.2-nlrr}
\lim_{T\to \infty}r(T)(\wth_{1,T}-\theta_1)\wc
-\frac{1}{p} \lim_{T\to \infty} r(T)(\wth_{2,T}-\theta_2)\wc
2\sqrt{2} p \sigma\, \eta.
\end{equation}

{\sc Roots of opposite sign.}
Assume that $q<0<p$ and let $r(T)$ be as in \eqref{rp}.
Then
\begin{equation}
\label{c4.3-nlrr}
\lim_{T\to \infty}r(T)(\wth_{1,T}-\theta_1)\wc
-\frac{1}{p} \lim_{T\to \infty} r(T)(\wth_{2,T}-\theta_2)\wc
\sigma\, \eta.
\end{equation}

{\sc  Zero root.}
Assume that $q<0$, $p=0$, and define
$$
r(T)= T^{-3/2}\int_0^TX^2(t)dt.
$$
Then
\begin{equation}
\label{c4.4-nlrr}
\lim_{T\to \infty}r(T)(\wth_{1,T}-\theta_1)\wc
\frac{\sigma^2}{\sqrt{2}|q|^{3/2}}\, \eta.
\end{equation}

{\sc Complex roots.} Assume that $p=\lambda+\sqrt{-1}\nu$,
$\lambda>0, \ \nu>0$ and define matrices
$$
A_T=
\left(
\begin{array}{lr}
\nu & 0 \\
\lambda & -1
\end{array}
\right)e^{-\lambda T},\
B(x,y)=
\frac{1}{x^2+y^2}\left(
\begin{array}{rr}
x & y \\
-y & x
\end{array}
\right).
$$
Then there exists a bivariate normal vector $(\mfu_s, \mfu_c)$ such that the
family
$$
B(\mfu_s,\mfu_c)A_T\Psi_T\,\big(\widehat{\bld{\theta}}_T-\bld{\theta}\big),\ T>0,
$$
is relatively compact, and all partial limits are bivariate normal random vectors
independent of $(\mfu_s, \mfu_c)$.
\end{theorem}

Theorem \ref{th3} suggests two negative conclusions:  (a) $\Psi_T^{1/2}$ is
usually not the correct random normalization (which is especially striking when
$p=\lambda+\sqrt{-1}\nu$); (b) existence of NLRR in non-ergodic CAR(2)
is not as helpful as in CAR(1),  because  the  rates and/or the limit
distributions  depend on the unknown parameters.

\vskip 0.1in

Finally, we describe the asymptotic structure of the normalized log-likelihood ratio \eqref{LAQ}.

\begin{theorem}[The structure of $\ell_{T}$]
\label{th4}
Denote by $p$ and $q$ the roots of equation \eqref{CharEq},
by ${\mathrm{diag}}(x,y)$ the diagonal matrix with $x$ and $y$ on
the main diagonal, and by $\bld{b}_p$ the column vector
$$
\bld{b}_p=
\left(
\begin{array}{rr}
1\\
p
\end{array}
\right).
$$

  If $\theta_1<0$, $\theta_2<0$ (ergodic case), and $A_T={\mathrm{diag}}(T^{-1/2}, T^{-1/2})$, then
 $\ell_{T}$ is LAN. \\

 If $p>0$, $q<p$, and $A_T=e^{-pT}\bld{b}_p\,\bld{b}_p^{\top}$, then
 $\ell_T$ is degenerate LAMN  and the matrix $B$ in \eqref{LAMN} is
  \begin{equation}
  \label{DLAMN-B}
  B=\frac{(1+p^2)^2}{2p}\sigma^2\zeta^2\,\bld{b}_p\,\bld{b}^{\top}_p,
  \end{equation}
   where
 $\zeta$ is a standard normal random variable. \\

 If $q<0=p$ and $A_T={\mathrm{diag}}(T^{-1}, T^{-1/2})$ then
 $\ell_T $ is mixed LABF/LAN:
 $$
 \ell_{\infty}(\bld{u})= \bld{u}^{\top}\bld{\xi}
  -\frac{1}{2}\bld{u}^{\top}B\bld{u},
  $$
  where
  $$
  \bld{\xi}=
  \left(
  \begin{array}{l}
  \sigma|q|^{-1}\int_0^1 w(s)dw(s)\\
 (2|q|)^{-1/2}\sigma\, \eta
  \end{array}
  \right),\
  B=
  \left(
  \begin{array}{cc}
  \sigma^2|q|^{-2}\int_0^1 w^2(s)ds & 0 \\
   0 & \sigma^2 (2|q|)^{-1}
   \end{array}
   \right),
   $$
 $\eta$ is a standard normal random variable, $w$ is a standard Brownian motion,
 and $\eta$ and $w$ are independent. \\

 If $p=q>0$ and
  $A_T=T^{-1}e^{-pT}\bld{b}_p\,\bld{b}_p^{\top}$, then
 $\ell_T$  is degenerate LAMN  and the matrix $B$ in \eqref{LAMN}
 is given by \eqref{DLAMN-B}. \\

 If $p=q=0$ and $A_T={\mathrm{diag}}(T^{-2}, T^{-1})$, then $\ell_T$
 is LABF and the matrix $G(t)$ in \eqref{LABF} is
 $$
G(t)=
\left(
\begin{array}{rl}
\sigma\int_0^tw(s)ds & 0\\
\sigma w(t) & 0
\end{array}
\right),
$$
where $w$ is a standard Brownian motions.  \\

If $p=\sqrt{-1}\nu$, $\nu>0$, and $A_T={\mathrm{diag}}(T^{-1}, T^{-1})$,
 then $\ell_T$  is  LABF  and the matrix $G(t)$ in \eqref{LABF} is
$$
G(t)=
\left(
\begin{array}{rl}
\sigma w_1(t) & \sigma w_2(t)\\
-\sigma w_2(t) & \sigma w_1(t)
\end{array}
\right),
$$
where $w_1$ and $w_2$ are independent standard Brownian motions.  \\

If $p=\lambda+\sqrt{-1}\nu$, $\lambda, \nu>0$, and
$$
A_T=
\left(
\begin{array}{lr}
\nu & 0 \\
\lambda & -1
\end{array}
\right) e^{-\lambda T},
$$
then $\{\ell_T,\ T>0\}$
is a relatively compact LAMN family having all partial limits  of the form \eqref{LAMN}.

\end{theorem}

One general conclusion of Theorem \ref{th4} is that, if $p>0$ and $q<p$,
then it is the value of the
larger root  $p$ that determines asymptotic behavior of the normalized log-likelihood ratio, which is
 in sharp contrast with Theorem \ref{th2}.

\section{Preparation for the proofs}
\label{Sec4}
To study asymptotic behavior of $\wth_{1,T}$
and $\wth_{2,T}$, we need the
expressions for the residuals $\wth_{i,T}-\theta_i$, $i=1,2$:
\begin{equation}
\label{residual}
\begin{split}
\widehat{\theta}_{1,T}-\theta_1&\!=\!
\frac{\left(\int_0^T{X}^2(t)dt\right)\left(\int_0^T\dot{X}(t)\sigma dW(t)\right)
-\left(\int_0^TX(t)\dot{X}(t)dt\right)\left(\int_0^T{X}(t)\sigma dW(t)\right)}
{\left(\int_0^T\dot{X}^2(t)dt\right)\left(\int_0^T{X}^2(t)dt\right)-
\left(\int_0^T{X}(t)\dot{X}(t)dt\right)^2},\\
\widehat{\theta}_{2,T}-\theta_2&\!=\!
\frac{\left(\int_0^T\dot{X}^2(t)dt\right)\left(\int_0^TX(t)\sigma dW(t)\right)-
\left(\int_0^TX(t)\dot{X}(t)dt\right)\left(\int_0^T\dot{X}(t)\sigma dW(t)\right)}
{\left(\int_0^T\dot{X}^2(t)dt\right)\left(\int_0^T{X}^2(t)dt\right)-
\left(\int_0^T{X}(t)\dot{X}(t)dt\right)^2}.
\end{split}
\end{equation}
These equalities follow directly from \eqref{MLE-main} and \eqref{eq00}.

Equation \eqref{eq00}  has a closed-form solution \eqref{intr-meq}, meaning that
  \eqref{residual} can be written in terms of integrals
of the type $\int_0^t f(s)dW(s)$. Unfortunately, this direct approach
quickly leads to intractable expressions. Therefore, despite availability
  of explicit formulas, a more sophisticated approach to the analysis of \eqref{residual}
  is necessary. In the ergodic case, the ergodic theorem provides all the necessary tools,
  and when $\theta_1=\theta_2=0$, the expressions are simplified using self-similarity
  of the standard Brownian motion. Other cases benefit from the following construction.

Given two square-integrable on $[0,T]$  functions $f,g$, define
\begin{align}
\label{Num}
N(T;f,g)&=\left(\int_0^T{f}^2(t)dt\right)\left(\int_0^Tg(t)\sigma dW(t)\right)\\
\notag
&-\left(\int_0^Tf(t)g(t)dt\right)\left(\int_0^Tf(t)\sigma dW(t)\right),\\
\label{Det}
D(T;f,g)&=\left(\int_0^T f^2(t)dt\right) \left(\int_0^T g^2(t)dt\right)
-\left(\int_0^T f(t)g(t)dt\right)^2.
\end{align}
Clearly, $D(T;f,g)=D(T;g,f)$, but in general
 $N(T;f,g)\not= N(T;g,f)$.
Then formulas \eqref{residual} become
\begin{equation}
 \label{residual1}
 \wth_{1,T}-\theta_1=\frac{N(T;X,\dX)}{D(T;X,\dX)},\ \
 \wth_{2,T}-\theta_2=\frac{N(T;\dX,X)}{D(T;X,\dX)}.
 \end{equation}

 The general idea of the proof of Theorem \ref{th2} is
to find the asymptotic behavior of $D(T;X,\dX)$, $N(T;X,\dX)$, and
$N(T;\dX,X)$, as $T\to \infty$. To keep track of the results, note that,
according to Table 3, $D(T;X,\dX)$ is
 dimensionless, $N(T;X,\dX)$ is measured in $[t]^{-1}$, and
 $N(T;\dX,X)$ is measured in $[t]^{-2}$.

For every real numbers $\alpha, \beta, \gamma, \kappa$
 and every square-integrable functions
$f,g$, we have the following identities:
\begin{align}
\label{Det-prop}
 &D(T; \alpha f+\beta g, \gamma f+\kappa g)=
  (\alpha \kappa-\beta\gamma)^2D(T; f,g),\\
 \label{N-prop}
&N(T; \alpha f+\beta g, \gamma f+\kappa g)
  = (\alpha^2\kappa-\alpha\beta\gamma)N(T;f,g)
 + (\beta^2\gamma-\alpha\beta\kappa)N(T;g,f).
\end{align}

Recall that  $\oas(t)$ denotes a continuous random process converging to
zero with probability one as $t\to \infty$. We will often use the
following result: if $f(t)>0$ is a continuous
process,  $F(T)=\int_0^T f(s)ds$, and $\lim_{T\to \infty} F(T)=\infty$ with
probability one, then
\begin{equation}
\label{TL}
\frac{1}{F(T)}\int_0^T f(t)\,\oas(t)dt=\oas(T).
\end{equation}
Indeed, \eqref{TL}  is immediate if the integral $\int_0^{\infty}  f(t)\,\oas(t)dt$
converges; otherwise, \eqref{TL} follows after one application of L'Hospital's rule.

Here are some other technical results to be used later.
For $r>0$, define Gaussian random variables
\bel{xi-eta}
\xi_r=\int_0^{\infty} e^{-rs} dW(s),\ \
\eta_r(T)=\int_0^Te^{-r(T-s)}dW(s).
\end{equation}
We have $\bE\eta_r(T)=0$, $\bE\eta_r^2(T)=(1-e^{-2rT})/(2r)$, and
therefore
\begin{equation}
\label{xi-eta1}
\lim_{T\to \infty} \eta_r(T)\wc \eta_r,
\end{equation}
where the random variable
 $\eta_r$ is normal  with mean zero and
variance $1/(2r)$. Since
\bel{indep}
\bE \xi_q\eta_r(T)= e^{-rT}\int_0^Te^{(r-q)s}ds \to 0, \ T\to \infty,
\ee
it follows that
$\eta_r$ and $\xi_q$ are independent for every $q,r>0$.

A more sophisticated version of the above observations is the
following result.
\begin{proposition}
\label{MLT}
Let $M=(M_1(t),\ldots,M_d(t)),\ 0\leq t\leq  1,$ be a $d$-dimensional continuous
Gaussian martingale with $M(0)=0$, and let
$M_T=(M_{T,1}(t),\ldots,M_{T,d}(t))$, $T\geq 0,\ 0\leq t\leq 1$,
 be a family of continuous
square-integrable $d$-dimensional martingales such that
$M_T(0)=0$ for all $T$  and, for every $t\in [0,1]$ and $i,j=1,\ldots,d$,
$$
\lim_{T\to \infty} \langle M_{T,i}, M_{T,j}\rangle (t)=\langle M_{i}, M_{j}\rangle (t)
$$
in probability. Then $\lim_{T\to \infty} M_T\wcL M$
in the topology of continuous functions on $[0,1]$.
\end{proposition}

\begin{proof} Modulo a non-essential (in this case) difference
between a sequence and a family indexed by the positive reals,
this is a particular case of Theorem VIII.3.11 in \cite{JSh}.
\end{proof}

For every $r>0$, the process $\eta_r(t), \ t\geq 0,$
is  ergodic (in fact, strictly mixing).   Therefore, the ergodic theorem
implies
\begin{equation}
\label{OU-LLN}
\frac{1}{T}\int_0^T \eta_r(t)dt=\oas(T),\ \
\frac{1}{T}\int_0^T \eta_r^2(t)dt =  \frac{1}{2r}+\oas(T),
\end{equation}
 and, together with Proposition \ref{MLT}
\begin{equation}
\label{OU-WC}
\begin{split}
&\lim_{T\to \infty} \frac{1}{\sqrt{T}}\int_0^T\eta_r(t)dW(t)\wc \eta_r^{\perp},
\\
&\lim_{T\to \infty} \left(\int_0^T \eta_r^2(t)dt\right)^{-1/2} \int_0^T\eta_r(t)dW(t)\wc \sqrt{2r}\ \eta_r^{\perp},
\end{split}
\end{equation}
 where $\eta_r^{\perp}$ is normal with mean zero and variance $1/(2r)$, and
 the random variables
$(\xi_q,\, \eta_r, \,\eta_r^{\perp})$ are jointly independent for every $q,r>0$.

Asymptotic analysis of certain stochastic integrals can benefit from the
law of iterated logarithm. Recall that if $f$ is locally square-integrable
adapted process, $M(t)=\int_0^t f(s)dW(s)$ and $\langle M\rangle(t)=\int_0^t f^2(s)ds\nearrow
+\infty$, $t\to \infty,$ with probability one, then there exists a
standard Brownian motion $\bar{W}$ such that
$$
M(t)=\bar{W}\big(\langle M\rangle(t)\big).
$$
The law of iterated
logarithm for $\bar{W}$ implies
\bel{LIL-gen}
\lim_{T\to \infty} \frac{M(T)}{\sqrt{\langle M\rangle(T)}\ (\ln \langle M\rangle(T))^{\varepsilon}}=\oas(T)
\ee
for every $\varepsilon >0$. For example, if $\varepsilon>0$, then
\bel{OU-LIL}
T^{-\varepsilon}\eta_r(T)=\oas(T).
\ee

To conclude the general discussion we establish an integration by parts formula.
As a motivation, recall that analysis of CAR(1) in the exponentially
unstable case leads to the function
 $V_p(t)=\int_0^t e^{p(t-s)}dW(s)$, $p>0$, for which
 integration by parts shows that
\bel{BP-00}
e^{-pT}\left(\int_0^T V_p(t)dW(t)-e^{pT}\xi_p\eta_p(T)\right)=\oas(T).
\ee
Since exponentially unstable
solutions of equation \eqref{DE-main} are of the form
$e^{pt}f(t)$, where the function $f$ grows at most polynomially,
we generalize \eqref{BP-00} as follows.

\begin{proposition}
\label{ByParts}
Given a deterministic (for simplicity) and locally square-integrable function
$f$, define $S_f(t)=\int_0^t f(s)dW(s)$.

Let functions $\varphi$ and $\psi$ be such that
\begin{equation}
\label{eq:BP1}
 \int_0^{+\infty} \psi^2(t)dt=+\infty,\ \ e^{pT} |\varphi(T)| +
e^{-pT} |\psi(T)|\leq C(1+T^r),\ T\geq 0,
\end{equation}
for some $p,C,r>0$.
Then
\begin{equation}
\label{eq:BP2}
e^{-qT}\left(\int_0^T S_{\varphi}(t)dS_{\psi}(t)-
S_{\varphi}(T)S_{\psi}(T)\right) = \oas(T) \ \ {\rm for \ all} \ \ q>0.
\end{equation}
\end{proposition}

\begin{proof}

By the It\^{o} formula,
\begin{equation}
\label{eq:BP3}
S_{\varphi}(T)S_{\psi}(T)- \int_0^T S_{\varphi}(t)dS_{\psi}(t)
 =
 \int_0^T \varphi(t)\psi(t)dt+ \int_0^TS_{\psi}(t)dS_{\varphi}(t),
\end{equation}
and it follows from \eqref{eq:BP1} that
$$
e^{-qT} \int_0^T|\varphi(t)\psi(t)|dt \leq C_1e^{-qT} (1+T^{r+1}) \to 0,\  T\to \infty.
$$
To estimate the second term on the right-hand side of \eqref{eq:BP3},
recall that, if $M=M(t)$ is a continuous square-integrable martingale, then,
by the strong law of large numbers, a finite limit
$$
\lim_{T\to \infty} \frac{M(T)}{1+\langle M\rangle (T)}
$$
exists with probability one (\cite[Corollary 2 to Theorem 2.6.10]{LSh3}).
As a result, if $F=F(t)$ is a function such that $F(T)\langle M\rangle (T)=\oas(T)$,
then $F(T)M(T)=\oas(T)$.

Next, consider
$$
N(t)= \int_0^tS_{\psi}(s)dS_{\varphi}(s)= \int_0^t S_{\psi}(s)\varphi(s)dW(s)\ \ {\rm with\ \ }
\langle N \rangle (t) = \int_0^t S_{\psi}^2(s)\varphi^2(s)ds.
$$
We need to show that $e^{-qT}\langle N \rangle (T)= \oas(T). $
By \eqref{eq:BP1},
\bel{BP4}
\langle S_{\psi}\rangle (t) = \int_0^t \psi^2(s)ds \leq C_2e^{2pt}(1+t^{2r}),
\ee
and therefore
\bel{BP5}
\lim_{t\to \infty} e^{-qt}\langle S_{\psi}\rangle (t)\varphi^2(t)\leq
\lim_{t\to \infty} C_3e^{-qt}(1+t^{3r})=0, \ q>0.
\ee
By assumption,   $\sup_t \langle S_{\psi}\rangle (t) =\infty$.
We then use \eqref{LIL-gen} and \eqref{BP4}  to conclude that
$$
\lim_{T\to \infty} \frac{Q^2(T)}{\langle S_{\psi}\rangle (T)\,  T^{\varepsilon}}=0, \ \varepsilon > 0.
$$
Therefore,
$$
 \lim_{T\to \infty} e^{-qT} S_{\psi}^2(T)\varphi^2(T)=
 \lim_{T\to \infty} T^{\varepsilon} e^{-qT}\langle S_{\psi}\rangle (T)\varphi^2(T)
 \frac{Q^2(T)}{\langle S_{\psi}\rangle (T)\,  T^{\varepsilon}} = 0,
 $$ which,
by L'Hospital's rule,  implies $\lim_{T\to \infty}e^{-pT}\langle N\rangle (T)=\oas(T)$.

{\em This completes the proof of Proposition \ref{ByParts}.}
\end{proof}

We will use Proposition \ref{ByParts} with $q=p$ to simplify various
stochastic integrals. As a quick illustration, let us verify \eqref{BP-00}.  Take $\varphi(t)=e^{-pt}$, $\psi(t)=e^{pt}$. Then, together with \eqref{xi-eta},
equality \eqref{eq:BP2} implies
\begin{equation*}
\begin{split}
e^{-pT}\int_0^T \left(\int_0^t e^{p(t-s)}dW(s)\right) dW(t)&=
\left(\int_0^T e^{-p(T-t)}dW(t)\right)\left(\int_0^T e^{-pt}dW(t)\right)
\\+
\oas(T)
&=\eta_p(T)\,\xi_p+\oas(T).
\end{split}
\end{equation*}

\section{Proofs of Theorems \ref{th2}--\ref{th4}}
\label{S4}

{\bf Proof of Theorem \ref{th2}}.

{\sc I. The ergodic case.}
See \cite[Remark 2 after Theorem 4.6.2]{Arato}.
 \vs
{\sc II. Non-ergodic case: Distinct real roots.}
If $p\not=q$, then \eqref{intr-meq} and \eqref{FS1}
imply
\begin{equation}
\label{sol-pq}
X(t)=V_p(t)- V_q(t),\ \
\dX(t)=pV_p(t)-qV_q(t),
\end{equation}
where
\begin{equation}
\label{sol-U}
\begin{split}
V_p(t)=e^{pt}U_p(t),\ U_p&=\frac{\dX(0)-X(0)q}{p-q}
+\frac{\sigma}{p-q} \int_0^t e^{-ps}dW(s),\\
V_q(t)=e^{qt}U_q(t),\ U_q&=\frac{\dX(0)-X(0)p}{p-q}
+\frac{\sigma}{p-q} \int_0^t e^{-qs}dW(s).
\end{split}
\end{equation}

By \eqref{Det-prop} and \eqref{N-prop}
\begin{equation}
\label{DN-X}
\begin{split}
 D(T;X,\dX)&=(p-q)^2D(T;V_p,V_q),\\
  N(T;X,\dX)&=(p-q)\Big(N(T,V_p,V_q)+N(T;V_q,V_p)\Big),\\
 N(T;\dX,X)&=-(p-q)\Big(pN(T,V_p,V_q)+qN(T;V_q,V_p)\Big).
 \end{split}
 \end{equation}
To complete the proof,
it now remains to use the specific expressions for the functions $V_p$ and $V_q$, which
we do next.

{\sc II(a). Roots of opposite sign:  $q<0<p$.}
 Using notations \eqref{xi-eta}, define  Gaussian random variable $\zeta_p$ by
$$
 \zeta_p=\frac{\dX(0)-X(0)q}{p-q}
 + \frac{\sigma}{p-q} \int_0^{+\infty} e^{-ps}dW(s)
 =\frac{\dX(0)-X(0)q}{p-q} + \frac{\sigma}{p-q}\xi_p
$$
  and note that
$$
 V_p(t)=e^{pt}\big( \zeta_p+\oas(t)\big),\ V_q(t)=\frac{\sigma}{p-q}\eta_{|q|}(t)
 +\frac{\dX(0)-X(0)p}{p-q}e^{qt}.
 $$
 Then
 \begin{align}
 \notag
 \int_0^T V_p^2(t)dt=\left(\frac{ \zeta_p^2}{2p}+\oas(T)\right)e^{2pT},\
 \int_0^T V_q^2(t)dt=\frac{\sigma^2T}{2|q|(p-q)^2}\big(1+\oas(T)\big),\\
 \notag
 \int_0^TV_p(t)V_q(t)dt=\int_0^Te^{pt}\big( \zeta_p+\oas(t)\big)V_q(t)dt
 =e^{pT}\big( \zeta_p\, \eta_{|q|}(T)+\oas(T)\big),\\
 \notag
  \int_0^T V_p(t)dW(t)=\sqrt{T}\,e^{pT}\,\oas(T),\ \ \
  \lim_{T\to \infty} \frac{1}{\sqrt{T}}\int_0^TV_q(t)dW(t)
  \wc\frac{\sigma}{p-q} \eta_{|q|}^{\perp}.
 \end{align}
 Plugging the results into \eqref{DN-X},
 \begin{equation*}
 \begin{split}
 D_T(T;X,\dX)=\left(Te^{2pT}\right) \left(\frac{ \zeta_p^2}{2p}\,\frac{\sigma^2}{2|q|}
 +\oas(T)\right),\\
 \lim_{T\to \infty} \frac{e^{-2pT}}{\sqrt{T}}N(T;V_p,V_q)
 \wc\frac{\sigma^2}{p-q}\frac{ \zeta_p^2}{2p}\eta_{|q|}^{\perp},\ \
 \lim_{T\to \infty} \frac{e^{-2pT}}{\sqrt{T}}N(T;V_q,V_p)\wc 0.
 \end{split}
 \end{equation*}
  Then both equalities in \eqref{c2.1}  follow from \eqref{residual1}.
\vs
{\sc II(b). Distinct positive roots: $0<q<p$.}
  When $p>q>0$, computations are very similar to the case $q<0<p$. The
  difference comes from the fact that, for $q>0$, we have
 $$
  V_q(t)=e^{qt}\left(\frac{\sigma}{p-q}\xi_q+\frac{\dX(0)-X(0)p}{p-q}+\oas(t)\right).
 $$
  Using Proposition \ref{ByParts} and  notations
  $$
  \zeta_p=\frac{\dX(0)-X(0)q}{p-q}+\frac{\sigma}{p-q}\,\xi_p,\
  \zeta_q=\frac{\dX(0)-X(0)p}{p-q}+\frac{\sigma}{p-q}\,\xi_q,
  $$
  we find
  \begin{align}
  \label{Det:p>q}
  D(T;X,\dX)= \frac{ \zeta_p^2\zeta_q^2 (p-q)^4}{4pq(p+q)^2} e^{2(p+q)T},\\
  \notag
 \lim_{T\to \infty} e^{-(q+2p)T}N(T;V_p,V_q)
 \wc\sigma\zeta_p^2\zeta_q\left(\frac{\eta_q}{2p}-\frac{\eta_p}{p+q}\right),\\
 \notag
 \lim_{T\to \infty}e^{-(q+2p)T}N(T;V_q,V_p)\wc 0.
 \end{align}
 It remains to observe that
 $$
 \frac{\eta_q}{2p}-\frac{\eta_p}{p+q}
 $$
 is a Gaussian random variable, independent of $(\zeta_p,\,\zeta_q)$,  with mean zero
 and variance $(p-q)^2/(8p^2q(p+q)^2)$.
  Then both equalities in \eqref{c4.1}  follow from \eqref{residual1}.

\vs
  {\sc II(c). Larger root is zero: $q<p=0$.}
 By \eqref{sol-pq} with $p=0$,
 \begin{equation}
 \label{th1-pr1}
 X(t)=V_0(t)-V_q(t),\ \dX(t)=-qV_q(t),
 \end{equation}
 and
 $$
 V_0(t)=U_0(t)=\frac{\dX(0)-X(0)q}{|q|}+\frac{\sigma}{|q|}W(t),\
 V_q=\frac{\dX(0)}{|q|}e^{qt}+\frac{\sigma}{|q|}\eta_{|q|}(t).
 $$

 Using \eqref{OU-LLN}, \eqref{OU-WC}, and \eqref{OU-LIL}, we find:
 \begin{align}
 \label{th1-pr3.1}
 \int_0^T V_q^2(t)dt= \frac{\sigma^2T}{2|q|^3}\left(1+\oas(T)\right),\\
 \label{th1-pr3.2}
 \int_0^T V_0^2(t)dt=\frac{\sigma^2T^2}{q^2}
 \left(\frac{1}{T^2}\int_0^T W^2(t)dt+\oas(T)\right),\\
 \label{th1-pr3.3}
 \int_0^T V_0(t)\dX(t)dt=T^{3/2}\oas(T),\\
 \label{th1-pr3.4}
 \lim_{T\to \infty} \frac{1}{\sqrt{T}}\int_0^TV_q(t)dW(t)\wc\frac{\sigma}{|q|} \eta_{|q|}^{\perp},\\
 \label{th1-pr3.5}
 \int_0^T U_0(t)dW(t)=\frac{\sigma T}{|q|}\left(\frac{W^2(T)-T}{2T} +\oas(T)
 \right).
 \end{align}
Self-similarity of the standard Brownian motion implies
 \begin{equation}
 \label{WSS}
 \frac{W^2(T)-T}{2T}\wc \frac{w^2(1)-1}{2},\ \
 \frac{1}{T^2}\int_0^T W^2(t)dt\wc \int_0^1 w^2(s)ds.
 \end{equation}

 Combining \eqref{th1-pr3.1}--\eqref{WSS} with \eqref{DN-X}
 yields
 \begin{align}
 \notag
 D(T;X,\dX)=\frac{\sigma^4 T^3}{2|q|^3}
 \left(\frac{1}{T^2} \int_0^T W^2(t)dt +\oas(T)\right),\\
 \notag
 \lim_{T\to \infty} T^{-5/2} N(T;X,\dX)
 \wc\frac{\sigma^4}{q^2}\eta_{|q|}^{\perp}
 \, \int_0^1 w^2(s)ds,\\
 \notag
 \lim_{T\to \infty} T^{-2}N(T;\dX,X)
 \wc\frac{\sigma^4}{4q^2}\big(w^2(1)-1\big).
 \end{align}
 Then both equalities in \eqref{c1.1}  follow from \eqref{residual1}.

 To show independence of $\eta$ and $w$, take a standard Brownian
 motion $\tilde{w}=\tilde{w}(t)$, $ t\in [0,1],$ that is independent of $w$ and
 apply  Proposition \ref{MLT} with
 $$
 M(t)=(w(t),\ \tilde{w}(t)),\ M_T(t)=\left(
 \frac{1}{\sqrt{T}}\int_0^{Tt}dW(s),\
 \frac{\sqrt{2|q|}}{\sqrt{T}}\int_0^{Tt}\eta_{|q|}(s)dW(s)\right).
 $$

{\sc II(d). Smaller root is zero: $q=0<p$.}
 We have
 \begin{align}
 \notag
 &X(t)=V_p(t)-V_0(t),\ \dX(t)=pV_p(t), \\
 \notag
 &V_0(t)=U_0(t)=\frac{\dX(0)-X(0)p}{p} + \frac{\sigma}{p}W(t),\
  V_p(t)=e^{pt}\big( \zeta_p+\oas(t)\big),\\
  \notag
 & \zeta_p=\frac{\dX(0)}{p}+\frac{\sigma}{p}\int_0^{\infty} e^{-pt}dW(t).
  \end{align}
  Then
  \begin{align}
 \notag
 \int_0^T V_p^2(t)dt= \frac{e^{2pT}}{2p}\left( \zeta_p^2+\oas(T)\right),&\
 \int_0^T V_0^2(t)dt=\frac{\sigma^2T^2}{p^2}
 \left(\frac{1}{T^2}\int_0^T W^2(t)dt+\oas(T)\right),\\
 \notag
 \int_0^T V_0(t)V_p(t)dt=Te^{pT}\oas(T),&\
 \int_0^TV_p(t)dW(t)= e^{pT}\big( \zeta_p\,\eta_{p}(T)+\oas(T)\big),\\
 \notag
 \int_0^T U_0(t)dW(t)&=\frac{\sigma T}{p}\left(\frac{W^2(T)-T}{2T} +\oas(T)
 \right).
 \end{align}
By \eqref{DN-X},
 \begin{align}
 \notag
 D(T;X,\dX)=\frac{\sigma^2 T^2e^{2pT}}{2p}
 \left( \zeta_p^2\,\frac{1}{T^2} \int_0^T W^2(t)dt +\oas(T)\right),\\
 \notag
  N(T;X,\dX)
 =\frac{\sigma^2}{2p}Te^{2pT}\left( \zeta_p^2\,\frac{W^2(T)-T}{2T} +\oas(T)
 \right),\\
 \notag
  N(T;\dX,X)
 =-\frac{\sigma^2}{2}Te^{2pT}\left( \zeta_p
 ^2\,\frac{W^2(T)-T}{2T} +\oas(T)
 \right),
 \end{align}
   and then both equalities in \eqref{c3.1} follow from \eqref{residual1}.
\vs
 {\sc III(a). Positive double root: $p=q>0$.}
With $x_1(t)=(1-qt)e^{qt}$, $x_2(t)=te^{qt}$,
 \eqref{intr-meq} becomes
 $$
 X(t)=X(0)(1-qt)e^{qt}+\dX(0)te^{qt}+\sigma\int_0^t (t-s)e^{q(t-s)}dW(s),\
 \dX(t)=qX(t)+Q(t),
 $$
 where
 $$
 Q(t)=\left(\dX(0)-X(0)q+\sigma\int_0^t e^{-qs}dW(s)\right)e^{qt}.
 $$
  If we define
 $$
 \zeta=\dX(0)-X(0)q+\sigma\int_0^{+\infty} e^{-qs}dW(s),
 $$
 then
 $$
 X(t)=te^{qt}\big(\zeta+\oas(t)\big),\ Q(t)=e^{qt}\big(\zeta+\oas(t)\big).
 $$
 By \eqref{Det-prop},
 \bel{pr5:D}
 D(T;X,\dX)=e^{4qT}\left(\frac{\zeta^4}{16q^4}+\oas(T)\right);
 \ee
 note that  the same result follows after passing to the
  limit $p\searrow q$ in \eqref{Det:p>q}.

Next, define
 $$
 \eta_{q,1}(t)=\int_0^te^{-q(t-s)}dW(s),\ \eta_{q,2}(t)
 =\frac{1}{t}\int_0^t s\,e^{-q(t-s)}dW(s),
 $$
 and observe that
 \begin{equation}
 \label{pr5:m}
 \lim_{T\to \infty} T\big(\eta_{q,1}(T)-\eta_{q,2}(T)\big)
 \wc \frac{\eta}{2q^{3/2}},
 \end{equation}
 where $\eta$ is a standard Gaussian random variable, independent of $\zeta$.

 Therefore,
 \begin{align}
 \notag
 \int_0^T X^2(t)dt = T^2e^{2qT}\left(\frac{\zeta^2}{2q}+\oas(T)\right),\ \ &
 \int_0^T Q^2(t)dt = e^{2qT}\left(\frac{\zeta^2}{2q}+\oas(T)\right),\\
 \notag
 \int_0^T X(t) Q(t)dt = Te^{2qt}\left(\frac{\zeta^2}{2q}+\oas(T)\right),\ \ &
 \int_0^T Q(t)dW(t)=e^{qT}\Big(\zeta\, \eta_{q,1}(T)+\oas(T)\big),\\
 \notag
 \int_0^T X(t)dW(t)&=Te^{qT}\Big(\zeta\, \eta_{q,2}(T)+\oas(T)\big).
 \end{align}
 To continue,
  \begin{align}
  \notag
  N(T;X,Q)=\frac{\sigma}{2q} T^2e^{3qT}
  \big(\eta_{q,1}(T)-\eta_{q,2}(T)\big)\big(\zeta^3+\oas(T)\big),\\
  \notag
  N(T;Q,X)=\frac{\sigma}{2q} Te^{3qT}
  \big(\eta_{q,2}(T)-\eta_{q,1}(T)\big)\big(\zeta^3+\oas(T)\big).
  \end{align}
  It remains to observe that
  $$
  N(T;X,\dX)=N(T;X,Q),\ N(T;\dX,X)=-qN(T;X,Q)+N(T;Q,X).
  $$
 Then both equalities in
 \eqref{c5.1}  follow from \eqref{residual1} and \eqref{pr5:m}.
\vs
{\sc III(b). Zero double root: $p=q=0$.}
In this case the result follows directly from \eqref{residual} using
 self-similarity of the standard
Brownian motion: $W(T\cdot)\wcL \sqrt{T}w(\cdot)$.
Recall the notations
$$
\mfz_1=\int_0^1 w(s)ds,\ \
\mfz_2=\int_0^1 w^2(s)ds,\ \
\mfz_3=\int_0^1\left( \int_0^t w(s)ds\right)^2 dt.
$$
With $\theta_1=\theta_2=0,$  equation  \eqref{eq00} becomes
$d\dX=\sigma\,dW$, and therefore
$$
\dX(t)=\dX(0)+\sigma\,W(t),\ X(t)=X(0)+\dX(0)t+\sigma\,\int_0^t W(s)ds.
$$
Then
\begin{align}
\notag
D(T;X,\dX)&\wc \frac{\sigma^4T^6}{4}\Big(4\mfz_2\mfz_3-\mfz_1^4+\oas(T)\Big),\\
\notag
N(T;X,\dX)&\wc
\frac{\sigma^4T^5}{2}\Big(\mfz_3\big(w^2(1)-1\big)-\mfz_1^2\big(w(1)\mfz_1-\mfz_2\big)+\oas(T)\Big),\\
\notag
N(T;\dX,X)&\wc
\frac{\sigma^4T^4}{4}\Big(4\mfz_2\big(w(1)\mfz_1-\mfz_2\big)-\mfz_1^2\big(w^2(1)-1\big)+\oas(T)\Big),
\end{align}
leading to \eqref{c6.1}. If $X(0)=\dX(0)=0$, then equalities in \eqref{c6.1}
hold for every $T>0$.
\vs
{\sc IV(a). Complex roots: $p=\sqrt{-1}\nu,\ \nu>0$. }
This case is the subject of \cite{Ning-Harmonic}. Below we outline the
main steps.

By \eqref{intr-meq},
\begin{equation*}
\begin{split}
X(t)=X(0)\cos(\nu t) +\frac{\dX(0)}{\nu}\sin(\nu t)
+\frac{\sigma}{\nu} \int_0^t \sin(\nu (t-s)) dW(s),\\
\dX(t)=-X(0)\nu\sin(\nu t) +{\dX(0)}\cos(\nu t)
+ \sigma\int_0^t \cos(\nu (t-s)) dW(s).
\end{split}
\end{equation*}
By Proposition \ref{MLT}, as $T\to \infty$, the pair
$$
\left(\frac{\sqrt{2}}{\sqrt{T}}\int_0^{Tt} \sin(\nu s)\, dW(s),\
\frac{\sqrt{2}}{\sqrt{T}}\int_0^{Tt} \sin(\nu s)\, dW(s),\ t\in [0,1]\right)
$$
converges in distribution  to a two-dimensional
standard Brownian motion $(w_1, w_2)$. Then
\begin{equation*}
\begin{split}
\lim_{T\to \infty} \frac{1}{T^2} \int_0^T X^2(t)dt \wc
 \frac{\sigma^2}{2\nu^2}\int_0^1(w_1^2(s)+w_2^2(s))ds,\\
\lim_{T\to \infty} \frac{1}{T^2} \int_0^T \dX^2(t)dt \wc
 \frac{\sigma^2}{2}\int_0^1(w_1^2(s)+w_2^2(s))ds,\\
 \lim_{T\to \infty} \frac{1}{T}\int_0^T X(t)dW(t) \wc
\frac{\sigma}{2\nu} \int_0^1 \big(w_1(s)dw_2(s)- w_2(s)dw_1(s)\big),\\
 \lim_{T\to \infty} \frac{1}{T}\int_0^T \dX(t)dW(t) \wc
\frac{\sigma}{2} \int_0^1 \big(w_1(s)dw_1(s)+ w_2(s)dw_2(s)\big).
\end{split}
\end{equation*}
Finally, \eqref{LIL-gen} implies $X^2(T)=T^2\oas(T)$. Then
both equalities in \eqref{HO-LD} follow from \eqref{residual}.
For details, see \cite[Sections 2--4]{Ning-Harmonic}.
\vs
{\sc IV(b). Complex roots: $p=\lambda+\sqrt{-1}\nu$,$\lambda>0$, $\nu>0.$}
The fundamental system of solutions of \eqref{DE-main} in this case is
$$
x_1(t)=\frac{e^{\lambda t}}{\nu}\left(\nu \cos\nu t-\lambda  \sin \nu t
\right),\ x_2(t)=\frac{1}{\nu} e^{\lambda t}\sin \nu t.
$$
Therefore,
\begin{align}
\label{CR-1}
X(t)&
=\frac{e^{\lambda t}}{\nu}\left((\dX(0)-X(0)\lambda)\sin\nu t+X(0)\nu \cos\nu t\right)
\\
\notag&+\frac{\sigma}{\nu}\int_0^t e^{\lambda (t-s)} \sin \nu (t-s)\, dW(s),\\
\label{CR-2}
\dX(t)&=\lambda X(t)+\nu Y(t),\\
\label{CR-21}
 Y(t)&=\frac{e^{\lambda t}}{\nu} \big((\dX(0)-X(0)\lambda)\cos \nu t-X(0) \nu \sin \nu t\big)\\
 & \notag+
\frac{\sigma}{\nu}\int_0^t e^{\lambda(t-s)}\cos\nu(t-s)dW(s).
\end{align}
Define Gaussian random variables
\begin{equation}
\label{CR-3}
\begin{split}
\mfu_c&=\frac{\dX(0)-X(0)\lambda}{\nu}+\frac{\sigma}{\nu} \int_0^{+\infty}
e^{-\lambda t} \cos\nu t\, dW(t),\\
\mfu_s&=-X(0)+\frac{\sigma}{\nu} \int_0^{+\infty}
e^{-\lambda t} \sin\nu t\, dW(t)
\end{split}
\end{equation}
and the functions
$$
V_c(t)=e^{\lambda t} \cos\nu t,\ V_s(t)=e^{\lambda t}\sin\nu t.
$$
Then
\begin{equation}
\label{CR-4}
X(t)=\mfu_cV_s(t)-\mfu_sV_c(t)+ e^{\lambda t}  \oas(t),\
Y(t)= \mfu_sV_s(t)+\mfu_cV_c(t)  +  e^{\lambda t} \oas(t).
\end{equation}
By \eqref{Det-prop},
\begin{equation}
\label{CR-5}
D(T;X,\dX)=\frac{e^{4\lambda T}\nu^4}{16\lambda^2(\lambda^2+\nu^2)}
\big((\mfu_c^2+\mfu_s^2)^2+\oas(T)\big).
\end{equation}

Next, define Gaussian random variables
\begin{equation}
\label{CR-6}
\mfh_c(T)=\int_0^T e^{\lambda(t-T)}\cos\nu t\, dW(t),\
\mfh_s(T)=\int_0^T e^{\lambda(t-T)}\sin\nu t\, dW(t).
\end{equation}
For each $T>0$, the vector $(\mfh_c(T), \mfh_s(T))$ is bivariate
normal with zero mean,
\begin{align}
\sigma_c^2(T)=\bE\mfh_c^2(T)&=\frac{1}{4\lambda} \left(
1+\frac{\lambda}{\sqrt{\lambda^2+\nu^2}}\cos(2\nu T-\phi)\right)+
o(1),\\
\sigma_s^2(T)=\bE\mfh_s^2(T)&=\frac{1}{4\lambda} \left(
1-\frac{\lambda}{\sqrt{\lambda^2+\nu^2}}\cos(2\nu T-\phi)\right)+o(1),\\
\sigma_{cs}(T)=\bE\mfh_c(T)\mfh_s(T)&= \frac{1}{4\sqrt{\lambda^2+\nu^2}}\sin(2\nu T - \phi)+
o(1),
\end{align}
where
$\cos\phi=\lambda/\sqrt{\lambda^2+\nu^2}$,
$\sin\phi=\nu/\sqrt{\lambda^2+\nu^2}$, and
$o(1)$ denotes a non-random function $\epsilon=\epsilon(T)$ such that
$\lim_{T\to \infty} \epsilon(T)=0$.
Thus, the family $\{\mfh_c(T), \mfh_s(T),\ T>0\}$ is relatively
compact, with limit points being bivariate Gaussian
vectors  $(\mfh_c,\mfh_s)$. Computations similar to \eqref{indep} show that
each vector $(\mfh_c,\mfh_s)$ is independent of $(\mfu_c,\mfu_s)$.

On the other hand,
\begin{equation}
\label{CR-N}
\begin{split}
N(T;V_c,V_s)&=\sigma e^{3\lambda T} \big(\sigma_c^2(T)\mfh_s(T)-\sigma_{cs}\mfh_c(T)\big),\\
N(T;V_s,V_c)&=\sigma e^{3\lambda T} \big(\sigma_s^2(T)\mfh_c(T)-\sigma_{cs}\mfh_s(T)\big).
\end{split}
\end{equation}
To complete the proof, it remains  to
express $N(T;X,\dX)$ and $N(T;\dX,X)$ in terms of
$N(T;V_c,V_s)$ and $N(T;V_s,V_c)$ using \eqref{N-prop} and
Proposition \ref{ByParts}.

{\em This completes the proof of Theorem \ref{th2}. }

\vskip 0.1in

{\bf Proof of Theorem \ref{th3}.}
 The proof is straightforward analysis of the computations
 in the proof of Theorem \ref{th2}, with the goal to find  suitable normalization leading
to the Gaussian limit. Let us illustrate this in the most interesting case, when
$p=\lambda+\sqrt{-1} \nu$ and the corresponding rate matrix $R_T$ is
not diagonal. In this case,  \eqref{CR-1}--\eqref{CR-4}
and \eqref{CR-6}, together with Proposition \ref{ByParts},  imply
\begin{equation}
\label{CR-main}
\begin{split}
\left(
\begin{array}{l}
\int_0^T X(s)dW(s)\\
\int_0^T \dX(s)dW(s)
\end{array}
\right)&=
-\frac{e^{\lambda T}}{\nu}
\left(
\begin{array}{rr}
-1 & 0\\
-\lambda & \nu
\end{array}
\right)
\left(
\begin{array}{rr}
\mfu_s & -\mfu_c\\
\mfu_c & \mfu_s
\end{array}
\right)
\left(
\begin{array}{l}
\mfh_c(T)+\oas(T)\\
\mfh_s(T)+\oas(T)
\end{array}
\right)\\
& = A_T^{-1}\big(B(\mfu_s,\mfu_c)\big)^{-1}
\left(
\begin{array}{l}
\mfh_c(T)+\oas(T)\\
\mfh_s(T)+\oas(T)
\end{array}
\right).
\end{split}
\end{equation}
On the other hand, according to \eqref{residual},
$$
\left(
\begin{array}{l}
\wth_{2,T}-\theta_2\\
\wth_{1,T}-\theta_1
\end{array}
\right)=
\Psi_T^{-1}
\left(
\begin{array}{l}
\int_0^T X(s)dW(s)\\
\int_0^T \dX(s)dW(s)
\end{array}
\right).
$$
That is,
$$
B(\mfu_s,\mfu_c)A_T\Psi_T\,\big(\widehat{\bld{\theta}}_T-\bld{\theta}\big)=
\left(
\begin{array}{l}
\mfh_c(T)+\oas(T)\\
\mfh_s(T)+\oas(T)
\end{array}
\right).
$$
{\em This completes the proof of Theorem \ref{th3}. }

\vskip 0.1in

{\bf Proof of Theorem \ref{th4}.}
 The proof is straightforward analysis of the computations
 in the proof of Theorem \ref{th2}, this time with an emphasis on
 the asymptotic behavior of the matrix $\Psi_T$ and the
 vector $\int_0^T\bld{X}(t)dW(t)$. The analysis is easy  in the ergodic case
 and also  when  $p=0$ or $p=\sqrt{-1}\nu$.
 When $p=\lambda+\sqrt{-1}\nu$, the proof is essentially complete
 after \eqref{CR-main}.

Here are the results when $p>0$.
If $q< p$, then
$$
\int_0^T\bld{X}(t)dW(t)= e^{pT}\zeta_p\big(\eta_p(T)+\oas(T)\big)\bld{b}_p,\
\Psi_T=\frac{\zeta^2_p}{2p}\,e^{2pT} \Big(\bld{b}_p\, \bld{b}_p^{\top}+\oas(T)\Big).
$$
If $q=p$, then
$$
\int_0^T\bld{X}(t)dW(t)=Te^{pT}\zeta\big(\eta_{p,2}(T)+\oas(T)\big)\bld{b}_p,\
\Psi_T=\frac{\zeta^2}{2p}\,T^2e^{2pT} \Big(\bld{b}_p\, \bld{b}_p^{\top}+\oas(T)\Big).
$$
Note also that
$$
\big(\bld{b}_p\, \bld{b}_p^{\top}\big)^2=(1+p^2)\bld{b}_p\, \bld{b}_p^{\top}.
$$
{\em This completes the proof of Theorem \ref{th4}. }

\section{Summary and Discussion}
\label{S5}
The maximum likelihood estimator \eqref{OU-MLE} in CAR(1), that is,
 the one-dimensional
OU process \eqref{OU}, has three types of asymptotic regimes, depending on the sign of
the parameter $\theta$. For CAR(2), that is, the second-order equation
 \eqref{eq00} with two unknown parameters, while the MLE still has the same form
 and is strongly consistent for all values of the parameters, the number of different
 asymptotic regimes is nine.
This jump in complexity underlines the challenges related to the analysis
of the general estimation problem, either for the $N$-th order
linear equation (CAR$(N)$) or an $N$-by-$N$ system.

If equation \eqref{CharEq} has real roots
and one of them is positive, then
the following two features seem to be common:
\begin{enumerate}
\item The  rate of convergence of the estimators
is determined by the smaller root of equation \eqref{CharEq};
\item The asymptotic behavior of the normalized log-likelihood
ratio is determined by the larger root of equation \eqref{CharEq}.
\end{enumerate}
It is interesting that the problem is mich easier for homogeneous equations, that is,
multi-dimensional analogues of the geometric Brownian motion.
In those models, estimation of the drift matrix in any number of dimensions
 leads to the LAN situation as long as the diffusion matrix is non-degenerate;
see \cite{KhasJan} for details.

The finite-difference equation arising from discretization of \eqref{eq00}
presents other challenges. In particular, the noise sequence driving the equation is
no longer independent; see \cite[Section 5]{Ning-Harmonic} for details.
While these difficulties can be resolved in the ergodic case (\cite{BDY-CAR}), the
general case remains unsolved.

\section{Acknowledgement} The work of both authors
 was partially supported
by the NSF Grant DMS-0803378.


\begin{thebibliography}{10}

\bibitem{Arato}
M.~Arat{\'o}, \emph{Linear stochastic systems with constant coefficients: a
  statistical approach}, Lecture Notes in Control and Information Sciences,
  vol.~45, Springer-Verlag, Berlin, 1982.

\bibitem{BasakLee}
G.~K. Basak and P.~Lee, \emph{Asymptotic properties of an estimator of the
  drift coefficients of multidimensional {O}rnstein-{U}hlenbeck processes that
  are not necessarily stable}, Electron. J. Stat. \textbf{2} (2008),
  1309--1344.

\bibitem{BasawaScott}
I.~V. Basawa and D.~J. Scott, \emph{Asymptotic optimal inference for nonergodic
  models}, Lecture Notes in Statistics, vol.~17, Springer-Verlag, New York,
  1983.

\bibitem{BDY-CAR}
P.~J. Brockwell, R.~A. Davis, and V.~Yang, \emph{Continuous-time {G}aussian
  autoregression}, Statistica Sinica \textbf{17} (2007), no.~1, 63--80.

\bibitem{Feigin}
P.~D. Feigin, \emph{Maximum likelihood estimation for continuous-time
  stochastic processes}, Advances in Appl. Probability \textbf{8} (1976),
  no.~4, 712--736.

\bibitem{Grw-Wfl}
P.~E. Greenwood and W.~Wefelmeyer, \emph{Asymptotic minimax results for
  stochastic process families with critical points}, Stochastic Process. Appl.
  \textbf{44} (1993), no.~1, 107--116.

\bibitem{IKh}
I.~A. Ibragimov and R.~Z. Khasminskii, \emph{Statistical
  estimation: {A}symptotic theory}, Applications of Mathematics, vol.~16,
  Springer, 1981.

\bibitem{JSh}
J.~Jacod and A.~N. Shiryaev, \emph{Limit theorems for stochastic processes, 2nd
  ed.}, Grundlehren der Mathematischen Wissenschaften, vol. 288, Springer,
  2003.

\bibitem{KhasJan}
A.~Jankunas and R.~Z. Khasminskii, \emph{Estimation of parameters of linear
  homogeneous stochastic differential equations}, Stochastic Process. Appl.
  \textbf{72} (1997), no.~2, 205--219.

\bibitem{Jeg-Gen}
P.~Jeganathan, \emph{Some aspects of asymptotic theory with applications to
  time series models}, Econometric Theory \textbf{11} (1995), no.~5, 818--887.

\bibitem{Kut2}
Yu.~A. Kutoyants, \emph{Statistical inference for ergodic diffusion processes},
  Springer, 2004.

\bibitem{LeCamY}
L.~Le~Cam and G.~L. Yang, \emph{Asymptotics in statistics: {S}ome basic
  concepts}, second ed., Springer Series in Statistics, Springer-Verlag, New
  York, 2000.

\bibitem{Ning-Harmonic}
N.~Lin and S.~V. Lototsky, \emph{Undamped harmonic oscillator driven by
  additive {G}aussian white noise: {A} statistical analysis}, Commun. Stoch.
  Anal. \textbf{5} (2011), no.~1, 233--250.

\bibitem{LSh3}
R.~Sh. Liptser and A.~N. Shiryaev, \emph{Theory of martingales}, Mathematics
  and its Applications (Soviet Series), vol.~49, Kluwer Academic Publishers,
  Dordrecht, 1989.

\bibitem{LSh1}
R.~Sh. Liptser and A.~N. Shiryaev, \emph{Statistics of random processes, {I}: {G}eneral theory, 2nd ed.},
  Applications of Mathematics, vol.~5, Springer, 2001.

\bibitem{Luschgy-LAMN}
H.~Luschgy, \emph{Local asymptotic mixed normality for semimartingale
  experiments}, Probab. Theory Related Fields \textbf{92} (1992), no.~2,
  151--176.

\end{thebibliography}

\def\cprime{$'$} \def\cprime{$'$} \def\cprime{$'$} \def\cprime{$'$}
  \def\cprime{$'$} \def\cprime{$'$}
\providecommand{\bysame}{\leavevmode\hbox to3em{\hrulefill}\thinspace}
\providecommand{\MR}{\relax\ifhmode\unskip\space\fi MR }
\providecommand{\MRhref}[2]{%
  \href{http://www.ams.org/mathscinet-getitem?mr=#1}{#2}
}
\providecommand{\href}[2]{#2}

\vskip 0.2in

\end{document}